\documentclass[12pt,a4paper]{amsart}

\usepackage[
  left=25truemm,
  right=25truemm,
  top=30truemm,
  bottom=30truemm
]{geometry}

\usepackage{mathtools}
\usepackage{amssymb}
\usepackage{amsthm}
\usepackage{thmtools}
\usepackage{mathrsfs}

\usepackage[dvipdfmx]{graphicx}
\usepackage{subcaption}
\usepackage{float}
\usepackage[dvipsnames]{xcolor}
\usepackage{framed}
\usepackage[all]{xy}
\usepackage{listings}
\usepackage{setspace}
\usepackage[english]{babel}
\usepackage{amscd}
\usepackage{enumitem}

\usepackage{tikz}
\usetikzlibrary{trees}

\usepackage{hyperref}
\usepackage{cleveref}

\makeatletter

\renewcommand{\section}{%
  \@startsection{section}{1}%
    {\z@}%
    {3ex}%
    {1ex}%
    {\normalfont\scshape\centering}%
}

\renewcommand{\subsection}{%
  \@startsection{subsection}{2}%
    {\z@}%
    {2ex}%
    {-.5em}%
    {\normalfont\bfseries}%
}

\makeatother

\newtheoremstyle{myplain}
  {10pt}                      
  {10pt}                      
  {\itshape}                  
  {}                          
  {\bfseries}                 
  {.}                         
  {8pt}                       
  {}                          

\newtheoremstyle{mydefinition}
  {10pt}
  {10pt}
  {\normalfont}
  {}
  {\bfseries}
  {.}
  {8pt}
  {}

\theoremstyle{myplain}
\newtheorem{thm}{Theorem}[section]
\newtheorem{lem}[thm]{Lemma}
\newtheorem{cor}[thm]{Corollary}
\newtheorem{prop}[thm]{Proposition}

\theoremstyle{mydefinition}
\newtheorem{example}[thm]{Example}
\newtheorem{defi}[thm]{Definition}
\newtheorem{rem}[thm]{Remark}

\newtheorem*{introthm1}{Theorem~\ref{thm3.4}}
\newtheorem*{introthm2}{Theorem~\ref{thm3.9}}
\newtheorem*{introthm3}{Theorem~\ref{thm3.14}}
\newtheorem*{introcor1}{Corollary~\ref{cor3.6}}
\newtheorem*{introcor2}{Corollary~\ref{cor3.17}}

\DeclareMathOperator{\Diff}{\mathop{\rm{Diff}}}
\DeclareMathOperator{\Int}{\mathop{\rm{Int}}}
\DeclareMathOperator{\rank}{\mathop{\rm{rank}}}
\Crefname{figure}{Figure}{Figures}

\numberwithin{equation}{section}

\title{\small On submersions with definite folds\\ of manifolds with boundary into Euclidean spaces}

\author[K. Iwakura]{Koki Iwakura}
\address{Joint Graduate School of Mathematics for Innovation, Kyushu University, Motooka 744, Nishiku, Fukuoka 819-0395, Japan.}
\email{iwakura.koki0105@gmail.com}

\date{}

\begin{document}

\maketitle

\begingroup
\renewcommand{\thefootnote}{}
\footnotetext{%
}
\endgroup

\begin{abstract}
Submersions with definite folds are submersions on manifolds with boundary whose restrictions to the boundary are definite fold maps. 
In this paper, we study differential-topological properties of manifolds with boundary admitting such maps into Euclidean spaces. 
When the target is $\mathbb{R}$, we obtain restrictions on the homeomorphism and diffeomorphism types of the source manifolds by using results for $m$-functions. 
For targets of dimension greater than one, we focus on the cases where the restrictions to the boundary are round fold maps or image simple fold maps, both defined by conditions on the singular value set. 
Then, we study the diffeomorphism types and Euler characteristics of manifolds admitting such maps. 
These results have applications to non-singular extensions of definite fold maps.
\end{abstract}

\section{Introduction}
It has been shown in several previous studies that singular points of maps on manifolds with boundary affect the global structures of the source manifolds, such as their topological types and smooth structures; see, for example, \cite{JR, Bra, Haj, BNR, SY, Yam}. 
However, many of these studies use arguments that depend essentially on the dimensions of the source and target manifolds. 
The present paper takes a complementary approach. 
Rather than imposing restrictions on these dimensions, we restrict the types of singular points admitted for maps. 
Following this approach, in \cite{Iwa3}, we studied boundary special generic maps on manifolds with boundary. 
A boundary special generic map is a submersion whose restriction to the boundary is a definite fold map, and it satisfies an additional condition on the behavior of a collar neighborhood of the boundary around the singular points.
This additional condition strongly constrains the structure of the source manifold.
Therefore, in order to clarify the relationship between singular points and a broader class of manifolds with boundary, it is natural to remove this condition.
In this paper, we introduce and study the resulting class of maps, called submersions with definite folds, which generalizes boundary special generic maps.

For submersions with definite folds, the methods developed for boundary special generic maps cannot be applied directly. 
Indeed, for boundary special generic maps, the Reeb space, which is the quotient space obtained by collapsing each connected component of each fiber, is known to be a manifold; see \cite{Iwa3}.
In \cite{Iwa3}, this property played a central role in studying the differential-topological properties of manifolds with boundary admitting boundary special generic maps.
By contrast, the Reeb spaces of submersions with definite folds need not be manifolds in general.
Thus, this broader class has to be studied from a different viewpoint. 
In this paper, we study submersions with definite folds in two cases: the case where the target is $\mathbb{R}$ and the case where the target has dimension greater than one. 
In the former case, we use known results on $m$-functions to study manifolds with boundary admitting such functions.
We prove that every compact connected surface with boundary admits submersions with definite folds into $\mathbb{R}$. 
In contrast, we show that, when the source manifold has dimension at least $3$, the existence of a submersion with definite folds imposes the following restrictions on its homeomorphism and diffeomorphism types.

\begin{introthm1}
Let $N$ be a compact, connected $n$-dimensional manifold with boundary, where $n\geq 3$.
If $N$ admits a submersion with definite folds $F\colon N\to\mathbb{R}$, then $N$ is homeomorphic to the manifold obtained by removing finitely many open balls from $S^n$. 
Moreover, $N$ is diffeomorphic to this manifold if and only if all but possibly one of the components of $\partial N$ are diffeomorphic to $S^{n-1}$.
\end{introthm1}

In the latter case, we focus on the restrictions to the boundary whose singular value sets satisfy additional conditions.
Specifically, we consider round fold maps or image simple fold maps. 
Let $F$ be a submersion with definite folds, where $N$ is a $3$-manifold with boundary and $F|_{\partial N}$ is a round fold map. 
Let $B_{1/2}\subset\mathbb{R}^2$ be a closed disk of radius $1/2$ centered at the origin, and set $L=\{(y_1,0)\in\mathbb{R}^2 \mid y_1\geq 1/2 \}$. 
Then, $F^{-1}(\mathbb{R}^2\setminus\Int B_{1/2})$ is a mapping torus with fiber $F^{-1}(L)$. 
We call the monodromy of this mapping torus the associated monodromy for $F$. 
We obtain the following result.

\begin{introthm2}
Let $N$ be a compact, connected, oriented $3$-dimensional manifold with boundary. 
If there exists a submersion with definite folds $F\colon N\to\mathbb{R}^2$ whose restriction to $\partial N$ is a round fold map, whose associated monodromy is periodic, such that $F^{-1}(0)$ has no $S^1$-components and $F^{-1}(L)$ is orientable, then $N$ is a graph manifold. 
\end{introthm2}

Furthermore, we consider a submersion with definite folds $F$ from a manifold $N$ with boundary to a Euclidean space, where $F|_{\partial N}$ is an image simple fold map. 
We define a tree $G_F$ from $F$. 
In terms of this tree, we obtain the following Euler characteristic formula.

\begin{introthm3}
Let $N$ be a compact, connected, $n$-dimensional manifold with boundary, and let $F\colon N\to\mathbb{R}^k$ be a submersion with definite folds whose restriction to the boundary is an image simple fold map, where $n>k$.
Then, we have 
$$
\chi(N)=\sum_{e\in E(0,-1)} \chi(S_e) + \sum_{e\in E(n-k,-1)} (-1)^{n-k}\chi(S_e)+ \sum_{e\in E} \bigl(\chi(R_e)-\chi(S_e)\bigr)\chi(F^{-1}(r_e)).
$$
Moreover, for each $e\in E$, we have 
$$
\chi(F^{-1}(r_e))=d_{(0,-1)}(r_e)-d_{(n-k,+1)}(r_e)+(-1)^{n-k-1}\bigl(d_{(0,+1)}(r_e)-d_{(n-k,-1)}(r_e)\bigr). 
$$
For the notation, see Section~3.2.2. 
\end{introthm3}

Our results on submersions with definite folds also have applications to the study of non-singular extensions of definite fold maps.
A non-singular extension of a given map on a closed manifold is a submersion whose restriction to the boundary realizes the given map.
The problem of determining when a given map on a closed manifold admits a non-singular extension is called the non-singular extension problem, and it has been studied in various contexts; see \cite{Cur, Sei, Iwa1, Iwa2, Iwa3}.
One of the main difficulties of this problem is that one has to deal simultaneously with both the singular points of the given map and those of its extension. 
Accordingly, many previous studies have focused on maps between manifolds of specific dimensions and have made essential use of properties specific to those dimensions.
Thus, many known results are available only in specific dimensional settings, and a more general picture is still not fully understood.
In \cite{Iwa3}, by using properties of boundary special generic maps, whose restrictions to the boundary are definite fold maps, we established several dimension-independent results on non-singular extensions of definite fold maps.
In this paper, we also obtain results on non-singular extensions of definite fold maps by using properties of submersions with definite folds.
For functions, we obtain the following result as a corollary of \Cref{thm3.4}.

\begin{introcor1}
Let $M$ be an $m$-dimensional closed manifold and let $f\colon M\to\mathbb{R}$ be a definite fold map, where $m\geq 2$. 
Suppose that $M$ is a disjoint union of homotopy spheres, at least one of which is exotic. 
If $f$ admits a non-singular extension, then at least two components of $M$ are exotic spheres. 
\end{introcor1}

Furthermore, when the target has general dimensions, we obtain the following result as a corollary of \Cref{thm3.14}.

\begin{introcor2}
Let $M$ be an $m$-dimensional closed manifold and let $f\colon M\to\mathbb{R}^k$ be an image simple fold map whose singular points are definite fold points, where $m\geq k$ and $m$ is even. 
If $f$ admits a non-singular extension $F\colon N\to\mathbb{R}^k$, then we have
\begin{align*}
\frac{1}{2}\chi(M)=\sum_{e\in E(0,-1)} \chi(S_e) + \sum_{e\in E(m-k,-1)} (-1)^{m-k+1}\chi(S_e)+ \sum_{e\in E} (\chi(R_e)-\chi(S_e))\chi(F^{-1}(r_e)).
\end{align*}
where the notation is the same as in \Cref{thm3.14}. 
\end{introcor2}

We now explain the organization of this paper. 
In Section~2, we review basic terminology from singularity theory of maps and introduce submersions with definite folds, which are the central objects of this paper.
Section~3 is devoted to the study of these maps.
In Section~3.1, we focus on the case where the target is $\mathbb{R}$.
In Section~3.2, we consider targets of dimension greater than one. 
We focus on the cases where the restrictions to the boundary are round fold maps or image simple fold maps.

Throughout this paper, all manifolds and maps are assumed to be of class $C^\infty$.
The boundary of an oriented manifold is oriented by the outward-first convention.

\section{Preliminaries}
In this section, we introduce submersions with definite folds, which are the main objects of this paper. 
We begin by reviewing singular points of maps. 
For more details, see \cite{GG}.

\begin{defi}[\textbf{Singular point}]\label{def2.1}
Let $M$ be an $m$-dimensional manifold, possibly with boundary, let $K$ be a $k$-dimensional manifold, where $m\geq k$, and let $f\colon M\to K$ be a map.  
A point $p\in M$ is called a \emph{singular point} of $f$ if $df_p$ is not surjective, that is, $\rank df_p<k$; otherwise, $p$ is called a \emph{regular point} of $f$. 
A point $q\in K$ is called a \emph{singular value} of $f$ if $q=f(p)$ for some singular point $p$ of $f$; otherwise $q$ is called a \emph{regular value} of $f$. 
In particular, $f$ is called a \emph{submersion} if every point of $M$ is a regular point of $f$. 
\end{defi}

The subset $S(f)$ of $M$ consisting of all singular points of $f$ is called the \emph{singular point set} of $f$; namely, 
$$
S(f)=\{p\in M\mid \rank df_p<k\}. 
$$
The image $f(S(f))\subset K$ is called the \emph{singular value set} of $f$.

Singular points of maps are often characterized by local normal forms. 
In this paper, we focus on definite fold points. 
We first recall fold points, of which definite fold points are a special case.

\begin{defi}[\textbf{Fold point}]\label{def2.2}
Let $M$ be an $m$-dimensional closed manifold, let $K$ be a $k$-dimensional manifold, where $m\geq k$, and let $f\colon M\to K$ be a map.  
A point $p\in S(f)$ is called a \emph{fold point} of $f$ if there exist local coordinates $(x_1,\dots,x_m)$ of $M$ around $p$ and $(y_1,\dots,y_k)$ of $K$ around $f(p)$ such that
$$
f; (x_1,\dots,x_m)\mapsto (y_1,\dots,y_k)=(x_1,\dots,x_{k-1},\sum_{i=k}^{m}\pm x_i^2), 
$$
where the signs for $x_i^2$ may be chosen independently. 
A map $f$ is called a \emph{fold map} if all points of $S(f)$ are fold points. 
A fold point is called a \emph{definite fold point} if all terms of the sum $\sum_{i=k}^m\pm x_i^2$ have the same sign. 
A fold map whose singular points are all definite fold points is called a \emph{definite fold map} (or a \emph{special generic map}). 
\end{defi}

We now introduce submersions with definite folds based on this definition. 
These maps are defined on manifolds with boundary.

\begin{defi}[\textbf{Submersion with definite folds}]\label{def2.3}
Let $N$ be a compact, connected, $n$-dimensional manifold with boundary, and let $K$ be a $k$-dimensional manifold, where $n>k$. 
A map $F\colon N\to K$ is called a \emph{submersion with definite folds} if $F$ is a submersion and $F|_{\partial N}$ is a definite fold map. 
\end{defi}

In this paper, we consider only submersions with definite folds whose target is a Euclidean space. 
The following proposition gives the local normal form of submersions with definite folds around a singular point of the restriction to the boundary. 
It is a special case of \cite[Lemma 3.3]{Iwa4}.

\begin{prop}\label{prop2.4}
Let $N$ be a compact, connected, $n$-dimensional manifold with boundary and let $F\colon N\to\mathbb{R}^k$ be a submersion with definite folds, where $n>k$. 
Then, for any point $p\in S(F|_{\partial N})$, there exist local coordinates $(x_1,\dots,x_n)$ of $N$ around $p$ and $(y_1,\dots,y_k)$ of $\mathbb{R}^k$ around $F(p)$ such that 
$$
F; (x_1,\dots,x_n)\mapsto (y_1,\dots,y_k)=(x_1,\dots,x_{k-1},\sum_{i=k}^{n-1}x_i^2\pm x_n), 
$$
where the regions $x_n>0$ and $x_n=0$ correspond to $\Int N$ and $\partial N$, respectively. 
\end{prop}

\begin{rem}\label{rem2.5}
(1) Let $F\colon N\to\mathbb{R}^k$ be a submersion, where $N$ is a compact $n$-dimensional manifold with $n>k$. 
Then, it is trivial that $F|_{\partial N}$ is a definite fold map if and only if $F$ has a local normal form around each singular point of $F|_{\partial N}$. 
Thus, we may equivalently define submersions with definite folds as submersions with the local normal forms given in \Cref{prop2.4}. 
This viewpoint shows that submersions with definite folds are a generalization of boundary special generic maps studied in \cite{Shi, Iwa3}. 
\\
(2) When $n-k=1$, a submersion with definite folds is simply a submersion whose restriction to the boundary is a fold map. 
\end{rem}

The results on submersions with definite folds can be applied to the study of non-singular extensions of definite fold maps. 
We next define this notion.

\begin{defi}[\textbf{Non-singular extension}]\label{def2.6}
Let $M$ be an $m$-dimensional closed manifold and let $f\colon M\to\mathbb{R}^k$ be a map, where $m\geq k$. 
Suppose that there exists a compact, connected $(m+1)$-dimensional manifold $N$ with $\partial N=M$ and a submersion $F\colon N\to\mathbb{R}^k$ that makes the following diagram commutative: 
$$
\begin{xy}
\xymatrix{
  M \ar[r]^{f} \ar[d]_{i} & \mathbb{R}^k \\
  N \ar[ru]_{F} &
}
\end{xy}
$$
where $i$ denotes the natural inclusion.
Then, $F$ is called a \emph{non-singular extension} of $f$. 
\end{defi}

Every submersion with definite folds is a non-singular extension of the definite fold map obtained by restricting it to the boundary.

\begin{rem}
In \cite{Iwa3}, we assume that the source manifold of the given map is connected. 
In contrast, no such assumption is needed in \Cref{def2.6}.
\end{rem}

\section{Properties of submersions with definite folds}
In this section, we study the properties of submersions with definite folds according to the dimension of the target. 
Throughout this section, let $N$ be a compact, connected, $n$-dimensional manifold with boundary and $F\colon N\to\mathbb{R}^k$ be a submersion with definite folds, where $n>k$.

\subsection{The case for Morse functions}
We first consider the case where $k=1$. 
In this case, submersions with definite folds are a special case of $m$-functions studied in \cite{Bra, Haj, JR}, and this fact is used later.   
As stated in \Cref{rem2.5}, when $n=2$, such functions are just submersions whose restrictions to the boundary are Morse functions. 
Thus, we divide the discussion into the cases $n=2$ and $n\geq 3$. 

\vspace{-10pt}

\subsubsection{The case $n=2$}
We first consider the case where $N$ is a surface with boundary. 
In this case, we show that the existence of a submersion with definite folds imposes no restriction on the topological type of $N$.

\begin{prop}\label{prop3.1}
Any compact, connected surface with boundary admits a submersion with definite folds into $\mathbb{R}$. 
\end{prop}

\begin{proof}
We construct such a submersion with definite folds from $N$ to $\mathbb{R}$, where $N$ is obtained by removing $b$ open disks from a closed surface given by the connected sum of $\Sigma_g$ and $s$ copies of $\mathbb{R}P^2$, where $g\geq 0$, $b\geq 1$, and $s\geq 0$. 
We consider the maps $F^{(j)}_{i}\colon N^{(j)}_i\to\mathbb{R}$ for $j=1,2,3$, $F^{(0)}\colon N^{(0)}\to\mathbb{R}$, and $F^{(4)}\colon N^{(4)}\to\mathbb{R}$ as \Cref{fig1}. 
In \Cref{fig1}, the number $i$ in (B) ranges from $0$ to $g-1$, $i$ in (C) ranges from $0$ to $s-1$, and $i$ in (D) ranges from $0$ to $b-1$. 
Arranging these maps according to the corresponding intervals in $\mathbb{R}$, we obtain a map 
$$
F\colon N^{(0)}\cup\Big(\bigcup_{i=0}^{g-1}N_i^{(1)}\Big)\cup\Big(\bigcup_{i=0}^{s-1}N_i^{(2)}\Big)\cup\Big(\bigcup_{i=0}^{b-1}N_i^{(3)}\Big)\cup N^{(4)}\to\mathbb{R}.
$$ 
By construction of $F$, the source manifold is diffeomorphic to $N$. 
Moreover, $F$ is a submersion, since it is obtained by gluing submersions corresponding to the intervals in $\mathbb{R}$ as above. 
Because all critical points of $F|_{\partial}$ are Morse critical points, $F$ is a submersion with definite folds. 
\begin{figure}[t]
\centering

\makebox[\textwidth][c]{
  \begin{subfigure}[b]{0.38\textwidth}
    \centering
    \includegraphics[width=\textwidth]{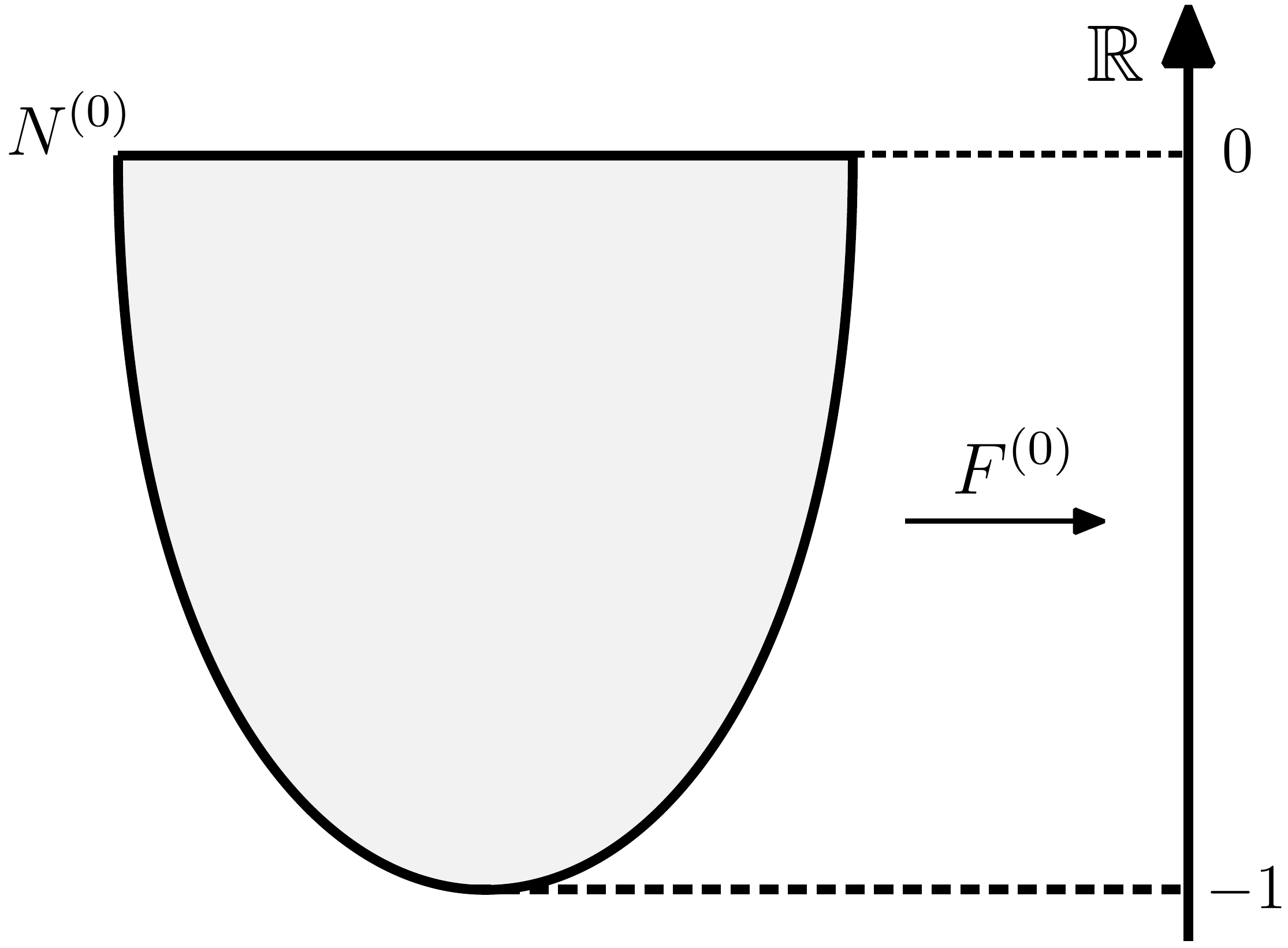}
    \caption{The map $F^{(0)}\colon N^{(0)}\to\mathbb{R}$.}
    \label{fig:1}
  \end{subfigure}
  \hspace{8mm}
  \begin{subfigure}[b]{0.4\textwidth}
    \centering
    \includegraphics[width=\textwidth]{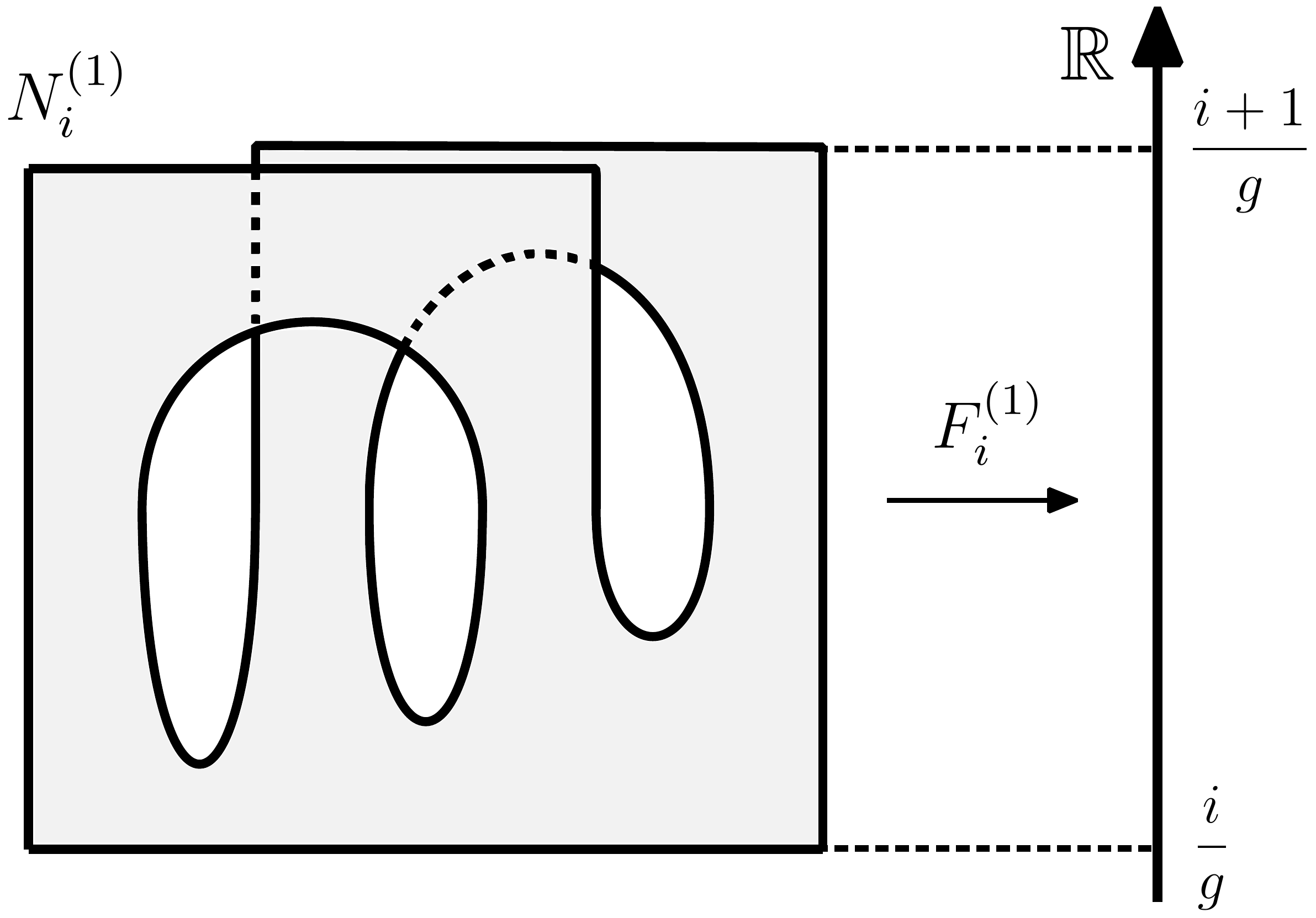}
    \caption{The map $F^{(1)}_i\colon N^{(1)}_i\to\mathbb{R}$.}
    \label{fig:2}
  \end{subfigure}
}

\vspace{3mm}

\makebox[\textwidth][c]{
  \begin{subfigure}[b]{0.38\textwidth}
    \centering
    \includegraphics[width=\textwidth]{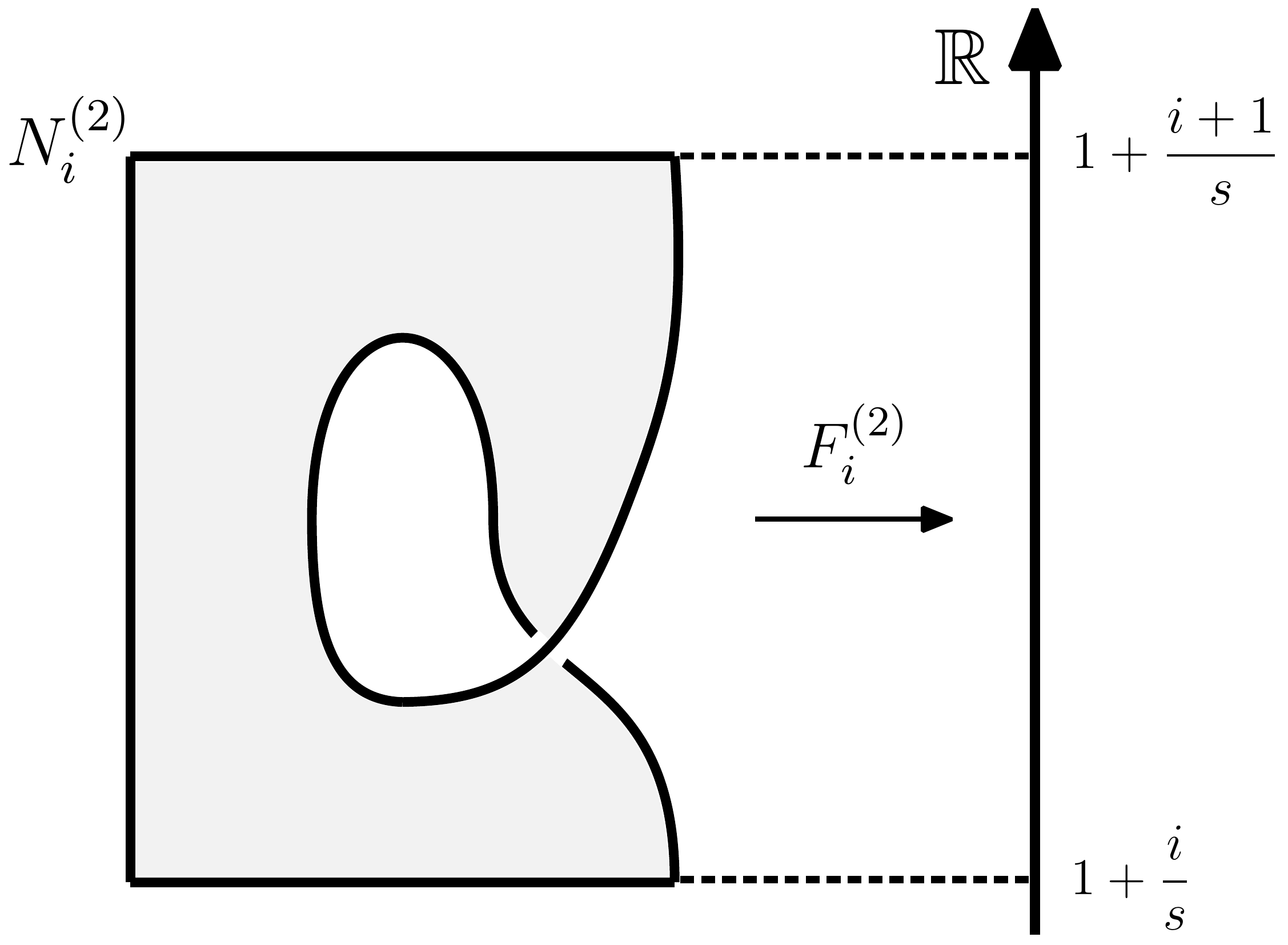}
    \caption{The map $F^{(2)}_i\colon N^{(2)}_i\to\mathbb{R}$.}
    \label{fig:3}
  \end{subfigure}
  \hspace{8mm}
  \begin{subfigure}[b]{0.39\textwidth}
    \centering
    \includegraphics[width=\textwidth]{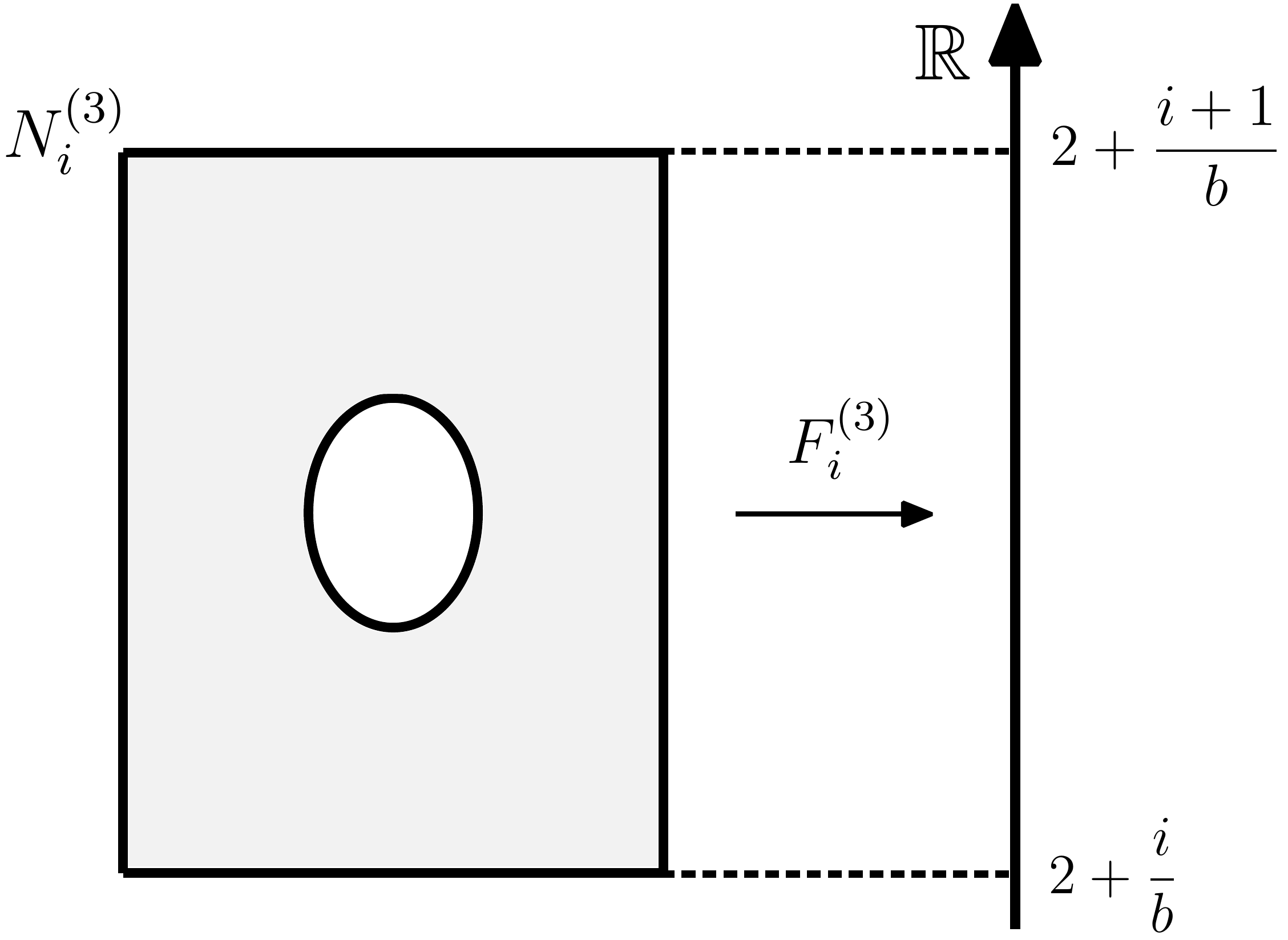}
    \caption{The map $F^{(3)}_i\colon N^{(3)}_i\to\mathbb{R}$.}
    \label{fig:4}
  \end{subfigure}
}

\vspace{3mm}

\makebox[\textwidth][c]{
  \begin{subfigure}[b]{0.38\textwidth}
    \centering
    \includegraphics[width=\textwidth]{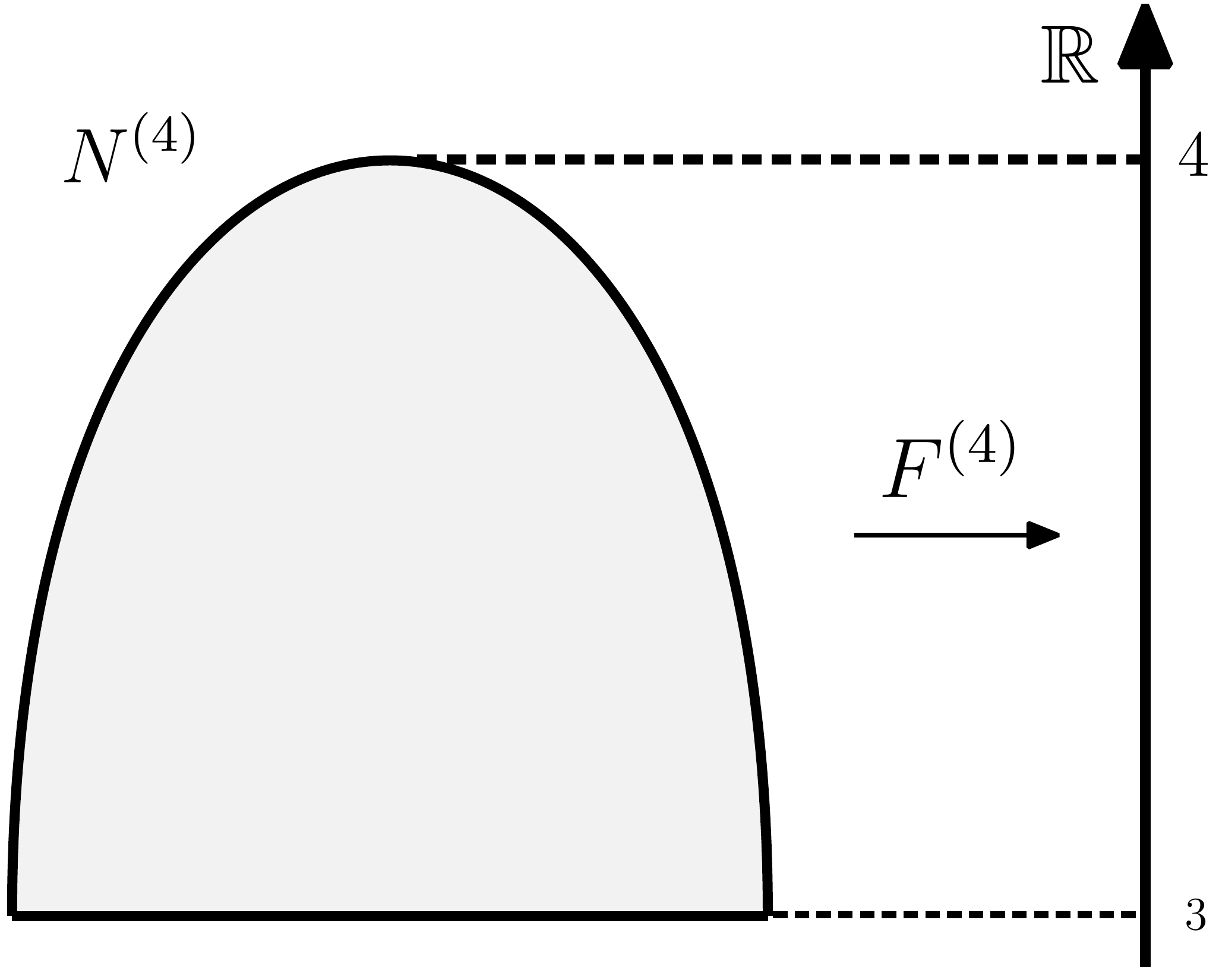}
    \caption{The map $F^{(4)}\colon N^{(4)}\to\mathbb{R}$.}
    \label{fig:5}
  \end{subfigure}
}

\caption{The parts to construct the submersion with definite folds $F\colon N\to\mathbb{R}$. 
}
\label{fig1}
\end{figure}
\end{proof}

If we impose restrictions on the topological type of the regular fibers of $F$, then submersions with definite folds impose restrictions on the topological type of $N$. 
Before stating our results, we introduce some notation.
Assume that $\mathbb{R}$ has its natural orientation. 
As \Cref{prop2.4}, around each $p\in S(F|_{\partial N})$, we take local coordinates $(x_1,\cdots,x_n)$ of $N$ around $p$ and $y$ of $\mathbb{R}$ around $F(p)$ that is consistent with the orientation so that 
$$
F;(x_1,\cdots,x_n)\mapsto y=\sum_{i=1}^{n-1} \pm x_i^2-\sigma x_n, 
$$
where $\sigma\in\{+1,-1\}$. 
Since $y$ is consistent with the orientation of $\mathbb{R}$,  the number of negative quadratic terms is independent of the choice of these coordinates.
We denote this number by $\lambda\in\{0,n-1\}$, and call $(\lambda,\sigma)$ by the \emph{index} of $p$.
We denote
$$
S^{\sigma}_{\lambda}(F)=\{p\in S(F|_{\partial N})\mid p\text{ is of index $(\lambda,\sigma)$}\}. 
$$
Note that $S(F|_{\partial N})=S^{-1}_{0}(F)\sqcup S^{-1}_{n-1}(F)\sqcup S^{+1}_{0}(F)\sqcup S^{+1}_{n-1}(F)$.

We recall some basic notions for $m$-functions used in the following arguments.
For details, see \cite{Haj, BNR}. 
As described in \cite[p.182, p.183]{Haj}, a manifold with boundary equipped with an $m$-function admits a decomposition called the $m$-decomposition. 
This is the same as the decomposition into (standard) handles and half-handles in \cite{BNR}.
Handles of indices $(k,-1)$ and $(k,+1)$ in Hajduk's terminology correspond to right half-handles of index $k$ and left half-handles of index $k+1$, respectively, in the terminology of Borodzik--N\'emethi--Ranicki.  
In this paper, we define $(\lambda,\sigma)$ to be compatible with that of Hajduk. 
Thus, the source manifold of a submersion with definite folds admits a decomposition consisting only of right half-handles of indices $0$ and $n-1$, and left half-handles of indices $1$ and $n-1$, with no (standard) handles. 
%
%

\begin{prop}
Let $N$ be a compact, connected, orientable surface with boundary.
Suppose that $N$ admits a submersion $F$ with definite folds such that each regular fiber is homeomorphic to $S^1$ or $D^1$ and the restriction $F|_{S(F|_{\partial N})}$ is injective.
Then $N$ is diffeomorphic to $\Sigma_{0,\sharp S_1^{-1}(F)+1}$.
Here, $\sharp$ denotes the cardinality of the relevant set.
\end{prop}

\begin{proof}
The assumption on the regular fibers of $F$ implies that, when the singular points of $F|_{\partial N}$ are ordered by increasing singular value, their indices occur as follows;  
$$
(0,-1),(1,-1),(0,+1),(1,-1),(0,+1),\ldots,(1,-1),(0,+1),(1,+1).
$$
Indeed, if there are singular points that do not follow this order, a fiber appears that is not homeomorphic to either $S^1$ or $D^1$.

We now consider how $N$ can be obtained by handle attachment.
For details, see \cite{Haj, BNR}.
Following \cite{BNR}, we use the notation $D^2=\{(x_1,x_2)\mid x_1^2+x_2^2\leq 1\}$, $D_+^2=D^2\cap\{x_1\geq 0\}$, $S_+^1=\partial D^2\cap\{x_1\geq 0\}$, $S_0^0=\partial D^2\cap\{x_1=0\}$, and $D_0^1=D^2\cap\{x_1=0\}$. 
A singular point with $(0,-1)$ corresponds to a right half-handle $D^0\times D_+^2$ of index $0$.
The next handle is attached to the interval $D^0\times D_0^1$, which we denote by $I$.
Next, a singular point with index $(1,-1)$ appears.
The corresponding right half-handle $D^1\times D_+^1$ of index $1$ is attached via an embedding $\psi\colon(\partial D^1\times D_+^1, \partial D^1\times D_0^0)\to(I,\partial I)$. 
After this attachment, the next handle is attached to the circle $(D^0\times D_0^1\setminus\psi(\partial D^1\times D_+^1))\cup_{\psi}(D^1\times S_+^0)$, which we denote by $S^1$.
Next, a singular point with $(0,+1)$ appears.
The corresponding left half-handle $D_+^1\times D^1$ is attached to $S^1$ via an embedding $\omega\colon S_+^0\times D^1\longrightarrow S^1$. 
After this attachment, the next handle is attached to the resulting interval.
These handle attachments corresponding to $(1,-1)$ and $(0,+1)$ are performed a total of $\sharp S_1^{-1}(F)$ times. 
Finally, we attach $D_+^2\times D^0$ to the remaining interval.
At each stage, the fiber to which the next handle is attached is an interval or a circle and is hence connected.
Moreover, $N$ is orientable.
Therefore, the above decomposition shows that $N$ is diffeomorphic to $\Sigma_{0,\sharp S_1^{-1}(F)+1}$.
\end{proof}

\begin{rem}
Conversely, for any $b\geq 1$, $\Sigma_{0,b}$ admits a submersion with definite folds into $\mathbb{R}$ whose regular fibers are $S^1$ or $D^1$. 
Such an example is shown in \Cref{fig2}. 
\begin{figure}[t]
  \centering
  \includegraphics[width=0.33\textwidth]{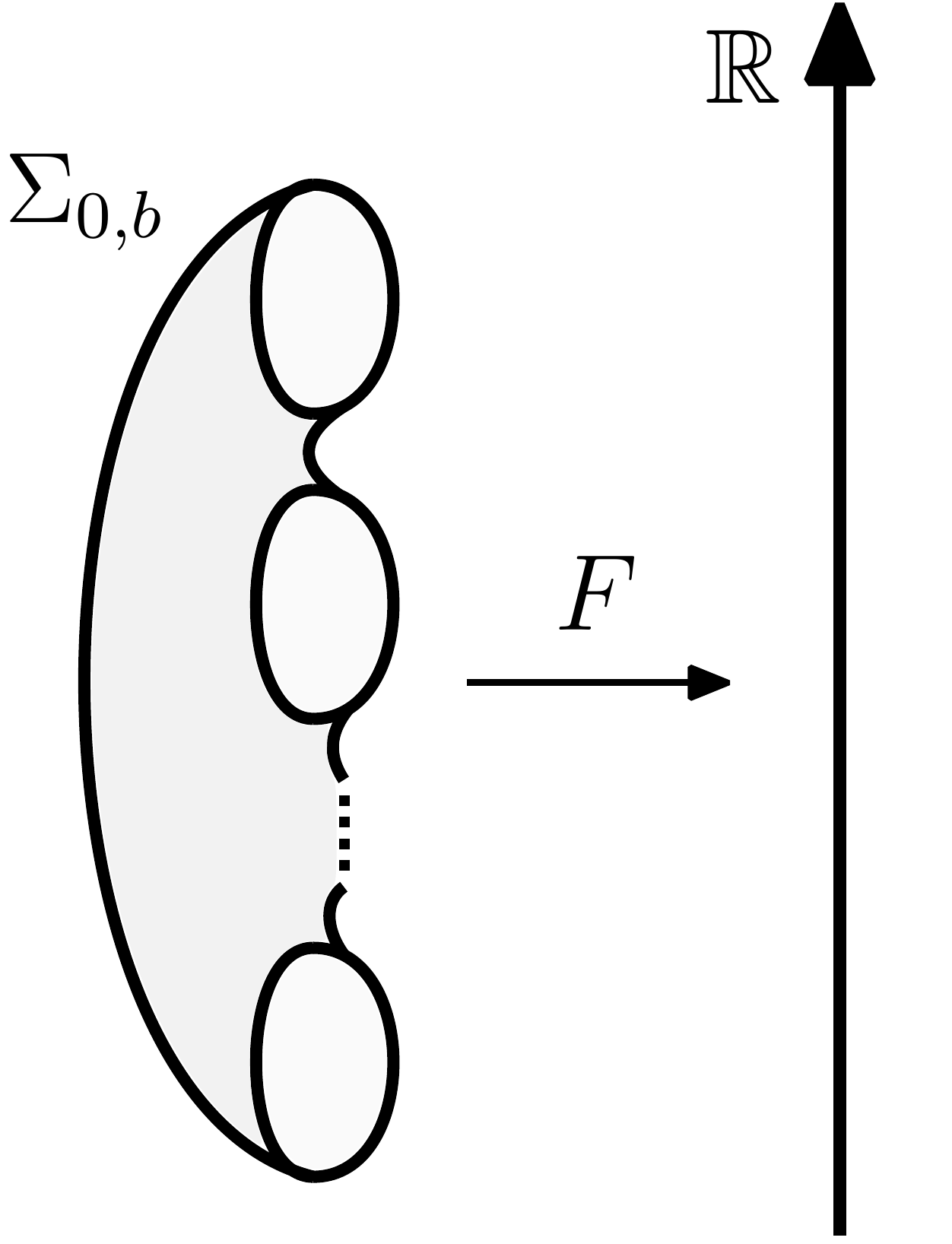}
  \caption{A submersion with definite folds $F\colon\Sigma_{0,b}\to\mathbb{R}$.}
  \label{fig2}
\end{figure}
\end{rem}

\subsubsection{The case $n\geq 3$}

We consider the homeomorphism type and the diffeomorphism type of a manifold admitting a submersion with definite folds into $\mathbb{R}$. 

\begin{thm}\label{thm3.4}
Let $N$ be a compact, connected $n$-dimensional manifold with boundary, where $n\geq 3$.
If there exists a submersion with definite folds $F\colon N\to\mathbb{R}$, then $N$ is homeomorphic to the manifold obtained by removing $(\sharp S_{n-1}^{-1}(F)+1)$ pairwise disjoint embedded $n$-dimensional open balls from $S^n$.
Moreover, $N$ is diffeomorphic to this manifold if and only if all but possibly one of the components of $\partial N$ are diffeomorphic to $S^{n-1}$.
\end{thm}

\begin{proof}
By \cite[Lemma~2.18, Lemma~2.19]{BNR}, $N$ is obtained by attaching $\sharp S_{n-1}^{-1}(F)$ $(n-1)$-handles to a $0$-handle.
We regard the $0$-handle as the complement of an $n$-dimensional open ball in $S^n$.
By the generalized Schoenflies theorem \cite[Theorem~5]{Bra} and the Alexander trick, each handle attachment is realized in the closure of one of the removed open balls, replacing one component with two components. 
Therefore, $N$ is homeomorphic to $S^n$ with $\sharp S_{n-1}^{-1}(F)+1$ $n$-dimensional open balls removed.

We now prove the criterion for the diffeomorphism.
If $N$ is diffeomorphic to $S^n$ with $\sharp S_{n-1}^{-1}(F)+1$ $n$-dimensional open balls removed, then $\partial N$ is diffeomorphic to a disjoint union of standard spheres.
We next prove the converse. 
Suppose that all but possibly one component of $\partial N$ are diffeomorphic to the standard sphere.
As above, $N$ is obtained by attaching $(n-1)$-handles to a $0$-handle.
We fill the $\sharp S_{n-1}^{-1}(F)$ boundary components that are diffeomorphic to the standard sphere with copies of $D^n$ and denote the resulting manifold by $\widetilde{N}$.
%
%
We cancel the $(n-1)$-handles against these $n$-handles in reverse order of attachment. 
At each step, at least one of the two boundary components created by the $(n-1)$-handle is capped off, and the corresponding $n$-handle forms a canceling pair.
Thus, $\widetilde{N}$ is diffeomorphic to $D^n$.
Therefore, removing the interiors of the disks from $\widetilde{N}$ gives the desired diffeomorphism. 
\end{proof}

\begin{rem}
Conversely, there exists a submersion with definite folds from the manifold obtained by removing open balls to $\mathbb{R}$. 
Such a map is constructed as a height function. 
\end{rem}

We now apply \Cref{thm3.4} to study non-singular extensions of definite fold maps into $\mathbb{R}$. 
The following corollary is an immediate consequence of \Cref{thm3.4}.

\begin{cor}\label{cor3.6}
Let $M$ be an $m$-dimensional closed manifold and let $f\colon M\to\mathbb{R}$ be a definite fold map, where $m\geq 2$. 
Suppose that $M$ is a disjoint union of homotopy spheres, at least one of which is exotic. 
If $f$ admits a non-singular extension, then at least two components of $M$ are exotic spheres. 
Moreover, the source manifold of the extension is homeomorphic, but not diffeomorphic, to $S^{m+1}$ with finitely many $(m+1)$-dimensional open balls removed. 
\end{cor}

%
%
%

\subsection{The case $k\geq 2$}
In this case, the general behavior of submersions with definite folds is more complicated. 
For this reason, we focus on the case where the restrictions to the boundary are round fold maps or image simple fold maps on closed manifolds.

\subsubsection{The case for round fold maps}
We consider submersions with definite folds from $3$-dimensional manifolds to $\mathbb{R}^2$ whose restrictions to the boundary are round fold maps. 
We first recall the definition of a round fold map; see \cite{Kit} for more details.

\begin{defi}[\textbf{Round fold map}]\label{def3.5}
Let $M$ be an $m$-dimensional closed manifold and let $f\colon M\to\mathbb{R}^k$ be a map, where $m\geq k$. 
The map $f$ is called a \emph{round fold map} if the following conditions are satisfied: 
\begin{enumerate}
\item $S(f)$ is diffeomorphic to a finite disjoint union of copies of $S^{k-1}$. 
\item $f|_{S(f)}$ is an embedding. 
\item $f(S(f))$ agrees with $\sqcup_{i=1}^{\sharp S(f)} S_i$, where $S_i$ is a $(k-1)$-dimensional sphere in $\mathbb{R}^k$ of radius $i$ centered at the origin. 
\end{enumerate}
\end{defi}

Below, we fix an orientation of the $3$-dimensional manifold $N$ and suppose that $F|_{\partial N}$ is a round fold map. 
We first decompose $\mathbb{R}^2$ as
$$
\mathbb{R}^2=B_{1/2}\cup(\mathbb{R}^2\setminus\Int B_{1/2}), 
$$
where $B_{1/2}=\{(y_1,y_2)\in\mathbb{R}^2 \mid y_1^2+y_2^2\leq(1/2)^2\}$. 
Corresponding to the decomposition of $\mathbb{R}^2$, we decompose $N$ as
\begin{equation}\label{equ1}
N=F^{-1}(\mathbb{R}^2\setminus\Int B_{1/2})\cup F^{-1}(B_{1/2}). 
\end{equation}
where $L=\{(y_1,0)\in\mathbb{R}^2 \mid y_1\geq 1/2 \}$.
Then, $F^{-1}(\mathbb{R}^2\setminus\Int B_{1/2})$ is an $F^{-1}(L)$-bundle over $S^1$. 
Thus, this manifold can be regarded as a mapping torus, and there exists a diffeomorphism $\phi\colon F^{-1}(L)\to F^{-1}(L)$ such that  
$$
F^{-1}(\mathbb{R}^2\setminus\Int B_{1/2})\cong [-1,1]\times F^{-1}(L)/(-1,x)\sim (1, \phi(x)). 
$$
We call $\phi$ the \emph{associated monodromy} of $F$.

\begin{thm}\label{thm3.9}
Let $N$ be a compact, connected, oriented $3$-dimensional manifold with boundary. 
If there exists a submersion with definite folds $F\colon N\to\mathbb{R}^2$ whose restriction to $\partial N$ is a round fold map, whose associated monodromy is periodic, such that $F^{-1}(0)$ has no $S^1$-components and $F^{-1}(L)$ is orientable, then $N$ is a graph manifold. 
\end{thm}

\begin{proof}
A mapping torus with a periodic diffeomorphism as its monodromy is a Seifert manifold.
Thus, $F^{-1}(\mathbb{R}^2\setminus\Int B_{1/2})$ is a Seifert manifold.
Moreover, $F^{-1}(B_{1/2})$ is diffeomorphic to a disjoint union of solid tori $S^1\times D^2$, one for each connected component of $F^{-1}(0)$.
These solid tori are attached to $F^{-1}(\mathbb{R}^2\setminus\Int B_{1/2})$ along their boundary tori.
Since a solid torus is a Seifert manifold, $N$ is a graph manifold.
\end{proof}

\begin{rem}
In \Cref{thm3.9}, we assume that $F^{-1}(0)$ has only $D^1$-components. 
If this assumption is removed, then the $S^1$-components of $F^{-1}(0)$ give rise to $3$-dimensional $2$-handles attached to the manifold considered above. 
The number of these handles is equal to the number of $S^1$-components of $F^{-1}(0)$. 
\end{rem}

\subsubsection{The case for image simple fold maps}
We consider submersions with definite folds whose restrictions to the boundary are image simple fold maps. 
In this case, we obtain a formula for computing the Euler characteristics of the source manifolds. 
The notation and results below generalize those of \cite{Iwa2}. 
We first recall the definition of image simple fold maps; see \cite{KS, SS} for their properties.

\begin{defi}[\textbf{Image simple fold map}]
Let $M$ be a closed $m$-dimensional manifold and let $f\colon M\to\mathbb{R}^k$ be a fold map with $m\geq k$. 
Then, $f$ is called an \emph{image simple fold map} if $f|_{S(f)}$ is an embedding. 
\end{defi}

\begin{rem}
A round fold map is an image simple fold map with additional properties. 
\end{rem}

In what follows, we assume that the restrictions of submersions with definite folds to the boundary are image simple fold maps. 
To give the formula, we first define the target graph $G_F$ as follows. 
We begin by defining an underlying graph $\widetilde{G}_F$.
The vertices of $\widetilde{G}_F$ correspond to the connected components of $\mathbb{R}^k\setminus F(S(F|_{\partial N}))$. 
Since $F|_{\partial N}$ is an image simple fold map, its restriction to the singular point set is an embedding.
Hence, its image is a hypersurface of $\mathbb{R}^k$, that is, a codimension one submanifold.  
We define an edge between two vertices corresponding to adjacent components of $\mathbb{R}^k \setminus F(S(F|_{\partial N}))$ along the singular value set.  
We denote by $\widetilde{G}_F$ the graph obtained in this way, consisting only of vertices and edges.
Then, by the definition of $\widetilde{G}_F$, it has the following property. 
The proof of \Cref{lem3.13} below immediately follows from the Jordan--Brouwer separation theorem~\cite[p.89]{GP}, so we omit it.

\begin{lem}\label{lem3.13}
The graph $\widetilde{G}_F$ is a tree, that is, it is connected and has no cycle.
\end{lem}

There exists a unique unbounded connected component of $\mathbb{R}^k\setminus F(S(F|_{\partial N}))$. 
Denote by $v_0$ the vertex corresponding to this region.  
By taking $v_0$ as the root, we may define the depth of each vertex as the number of edges in the path from $v_0$ to the given vertex.  
We direct each edge so that the depth increases along it.  
Moreover, with respect to the direction in which the depth increases, we take a line $L$ transverse to the corresponding component of $F(S(F|_{\partial N}))$. 
Then, the composition $\pi_L\circ F$, where $\pi_L\colon\mathbb{R}^k\to L$ is the projection, is again a submersion with definite folds. 
Since the line $L$ is oriented in the direction in which the depth increases, each edge is labeled by one of four types
$$
(0,-1),\ (0,+1),\ (n-k,-1),\ (n-k,+1), 
$$
as in Section~3.1. 
Thus, by assigning to each edge a direction and a label $(\lambda,\sigma)\in\{0,n-k\}\times\{+1,-1\}$, and to each vertex its depth, we obtain a graph $G_F$.  
We call this graph the \emph{target graph} of $F$.  
Here, for each $(\lambda,\sigma)$, let $E(\lambda,\sigma)$ denote the set of all edges of $G_F$ labeled by $(\lambda,\sigma)$.  
Note that, denoting by $E$ the set of the edges of $G_F$, we have 
\begin{equation}
E=E(0,-1)\sqcup E(0,+1)\sqcup E(n-k,-1)\sqcup E(n-k,+1).
\end{equation}
Then, we obtain the following theorem.

\begin{thm}\label{thm3.14}
Let $N$ be a compact, connected, $n$-dimensional manifold with boundary, and let $F\colon N\to\mathbb{R}^k$ be a submersion with definite folds whose restriction to the boundary is an image simple fold map, where $n>k$.
Then, we have 
$$
\chi(N)=\sum_{e\in E(0,-1)} \chi(S_e) + \sum_{e\in E(n-k,-1)} (-1)^{n-k}\chi(S_e)+ \sum_{e\in E} \bigl(\chi(R_e)-\chi(S_e)\bigr)\chi(F^{-1}(r_e)).
$$
Moreover, for each $e\in E$, we have 
$$
\chi(F^{-1}(r_e))=d_{(0,-1)}(r_e)-d_{(n-k,+1)}(r_e)+(-1)^{n-k-1}\bigl(d_{(0,+1)}(r_e)-d_{(n-k,-1)}(r_e)\bigr). 
$$
Here, $R_e$ is the component of $\mathbb{R}^k\setminus F(S(F|_{\partial N}))$ corresponding to the head of $e$, $S_e$ is the component of $F(S(F|_{\partial N}))$ corresponding to $e$, $r_e$ is a point of $R_e$, and $d_{(\lambda,\sigma)}(r)$ is the number of edges with $(\lambda,\sigma)$ on the path from $v_0$ to the vertex corresponding to the component containing $r$. 
\end{thm}

\begin{proof}
In terms of $G_F$, we decompose $N$ as
$$
N=\bigcup_{e\in E}\bigl(F^{-1}(R_e)\cup F^{-1}(N(S_e))\bigr), 
$$
where $N(S_e)\subset\mathbb{R}^k$ is a tubular neighborhood of $S_e$. 
By \Cref{lem3.13}, $G_F$ is a tree with $v_0$ as its root.

Therefore, starting with $F^{-1}(R_0)=\emptyset$, we obtain $N$ as in the above decomposition by successively attaching each $F^{-1}(R_e)$ to the part already constructed via $F^{-1}(N(S_e))$, following the edges of $G_F$ away from the root.
Here, $R_0$ is the unbounded region corresponding to the root.
Since $G_F$ has no cycles, there is a unique path from the root to each vertex.
Thus, by the definition of this decomposition, each newly attached piece intersects the part already constructed only along the attaching region corresponding to the tail of the edge under consideration.
Therefore, to compute $\chi(N)$, it suffices to calculate the change in the Euler characteristic at each attachment and sum these local contributions.

Let $Q$ denote the manifold obtained by performing the attachment starting at vertex $v_0$ and ending at the vertex corresponding to the tail of $e$. 
For any $e\in E$, we have 
\begin{align*}
\chi\bigl(Q\cup F^{-1}(N(S_e))\cup F^{-1}(R_e)\bigr)=
& \chi\bigl(Q\cup F^{-1}(N(S_e))\bigr)+\chi(F^{-1}(R_e))\\
& -\chi\big((Q\cup F^{-1}(N(S_e)))\cap F^{-1}(R_e)\big). 
\end{align*}
Here, $F^{-1}(R_e)$ is an $F^{-1}(r_e)$-bundle over $R_e$. 
Moreover, $F^{-1}(R_e)\cap F^{-1}(N(S_e))$ is an $F^{-1}(r_e)$-bundle over $R_e\cap N(S_e)$. 
Since $R_e\cap N(S_e)$ deformation retracts onto $S_e$, by the multiplicativity of the Euler characteristic for fiber bundles, we have
\begin{align*}
\chi\bigl(Q\cup F^{-1}(N(S_e))\cup F^{-1}(R_e)\bigr)=
\bigl(\chi(R_e) - \chi(S_e) \bigr) \chi(F^{-1}(r_e))+ \chi\bigl(Q\cup F^{-1}(N(S_e))\bigr).
\end{align*}

Independently of the type of $e$, we now have 
$$
\chi\bigl(Q\cup F^{-1}(N(S_e))\bigr)=\chi(Q)+
\chi(F^{-1}(N(S_e)))-\chi\bigl(Q\cap F^{-1}(N(S_e))\bigr). 
$$
We further compute $\chi\bigl(Q\cup F^{-1}(N(S_e))\bigr)$ according to the type of $e$.

Suppose that $e\in E(0,-1)$. 
Let $\rho_e\colon N(S_e)\to S_e$ be the projection. 
By the local normal form of $F$, $\rho_e\circ F|_{F^{-1}(N(S_e))}\colon F^{-1}(N(S_e))\to S_e$ is a $D^{n-k+1}$-bundle over $S_e$.
Moreover, $Q\cap F^{-1}(N(S_e))=\emptyset$. 
Thus, by the multiplicativity of the Euler characteristic for fiber bundles, we have 
$$
\chi\bigl(Q\cup F^{-1}(N(S_e))\bigr)=\chi(Q)+\chi(S_e). 
$$

Suppose next that $e\in E(n-k,-1)$. 
Again, let $\rho_e\colon N(S_e)\to S_e$ be the projection. 
By the local normal form of $F$ in \Cref{prop2.4}, $\rho_e\circ F|_{F^{-1}(N(S_e))}\colon F^{-1}(N(S_e))\to S_e$ and $\rho_e\circ F|_{Q\cap F^{-1}(N(S_e))}\colon Q\cap F^{-1}(N(S_e))\to S_e$ are fiber bundles. 
Let $A$ and $B$ denote their fibers, respectively. 
The local normal form shows that $A$ is obtained from $B$ by attaching an $(n-k+1)$-dimensional $(n-k)$-handle. 
Hence, 
$$
\chi(A)-\chi(B)=(-1)^{n-k}. 
$$
Therefore, by the multiplicativity of the Euler characteristic for fiber bundles, we have 
$$
\chi\bigl(Q\cup F^{-1}(N(S_e))\bigr)=\chi(Q)+(-1)^{n-k}\chi(S_e). 
$$

Finally, suppose that $e\in E(0,+1)\cup E(n-k,+1)$. 
Then, $Q\cup F^{-1}(N(S_e))$ deformation retracts onto $Q$.
It follows from the local normal forms in \Cref{prop2.4}. 
Thus, 
$$
\chi\bigl(Q\cup F^{-1}(N(S_e))\bigr)=\chi(Q). 
$$

We compute $\chi(N)$ as follows.
Regardless of the type of $e$, we have
\begin{align*}
\chi\bigl(Q\cup F^{-1}(N(S_e))\cup F^{-1}(R_e)\bigr)=&
\chi(Q)-\chi\bigl(Q\cup F^{-1}(N(S_e))\bigr)\\
&+\chi\bigl(F^{-1}(N(S_e))\bigr)+\bigl(\chi(R_e)-\chi(S_e)\bigr)\chi\bigl(F^{-1}(r_e)\bigr).
\end{align*}
We compute $\chi(N)$ by repeatedly applying this equation along the edges of $G_F$.
Since $G_F$ is a tree, from the above computations for each type of $e$, we obtain
\begin{align*}
\chi(N)
&=
  \sum_{e\in E(0,-1)}
  \Bigl(
    \chi(S_e)
    +\bigl(\chi(R_e)-\chi(S_e)\bigr)\chi\bigl(F^{-1}(r_e)\bigr)
  \Bigr)\\
&\quad+
  \sum_{e\in E(n-k,-1)}
  \Bigl(
    (-1)^{n-k}\chi(S_e)+\bigl(\chi(R_e)-\chi(S_e)\bigr)\chi\bigl(F^{-1}(r_e)\bigr)
  \Bigr)\\
&\quad+
  \sum_{e\in E(0,+1)\cup E(n-k,+1)}
  \bigl(\chi(R_e)-\chi(S_e)\bigr)
  \chi\bigl(F^{-1}(r_e)\bigr),
\end{align*}
which implies the desired equation.

We next prove the second equation.
We examine how the regular fiber of $F$ changes when one crosses a component of the singular value set of $F|_{\partial N}$.
For $e\in E$, denote by $r_e$ a point in the region corresponding to the head of $e$ and by $t_e$ a point in the region corresponding to the tail of $e$.
Then, we compute $\chi(F^{-1}(r_e))-\chi(F^{-1}(t_e))$ as follows.
When $e$ is labeled with $(0,-1)$, the topological type of the regular fiber changes by adding one disk component; that is, $\chi(F^{-1}(r_e))-\chi(F^{-1}(t_e))=1$.
When $e$ is labeled with $(n-k,-1)$, it changes by filling in a boundary component diffeomorphic to $S^{n-k-1}$ with $D^{n-k}$; that is, $\chi(F^{-1}(r_e))-\chi(F^{-1}(t_e))=(-1)^{n-k}$.
When $e$ is labeled with $(0,+1)$, it changes by removing $\Int D^{n-k}$; that is, $\chi(F^{-1}(r_e))-\chi(F^{-1}(t_e))=(-1)^{n-k-1}$.
When $e$ is labeled with $(n-k,+1)$, it changes by deleting one disk component; that is, $\chi(F^{-1}(r_e))-\chi(F^{-1}(t_e))=-1$.
To compute $\chi(F^{-1}(r_e))$, it suffices to trace these changes along the unique path in $G_F$ from $v_0$ to the vertex corresponding to $R_e$, since $G_F$ is a tree.
Therefore, by summing these equations along this path, we obtain the second formula.
\end{proof}

We apply this formula for an example.

\begin{example}
Let $i\colon S^1\times D^2\to\mathbb{R}^3$ be an embedding and $\pi\colon\mathbb{R}^3\to\mathbb{R}^2$ be a standard projection.
Then, we define a submersion with definite folds $F$ by $\pi\circ i$; see \Cref{Fig3.1}. 
The target graph $G_F$ is shown in \Cref{Fig3.2}. 
Since $S_e$ is homeomorphic to $S^1$ for each edge $e\in E$, by \Cref{thm3.14}, the Euler characteristic of the source manifold of $F$ is given by
$$
\sum_{e\in E}\chi(R_e)\chi(F^{-1}(r_e)). 
$$
An easy computation shows that this is $0$, giving $\chi(S^1\times D^2)=0$.
\begin{figure}
  \centering
  \makebox[\textwidth][c]{%
    \begin{subfigure}[b]{0.43\textwidth}
      \centering
      \includegraphics[width=\textwidth]{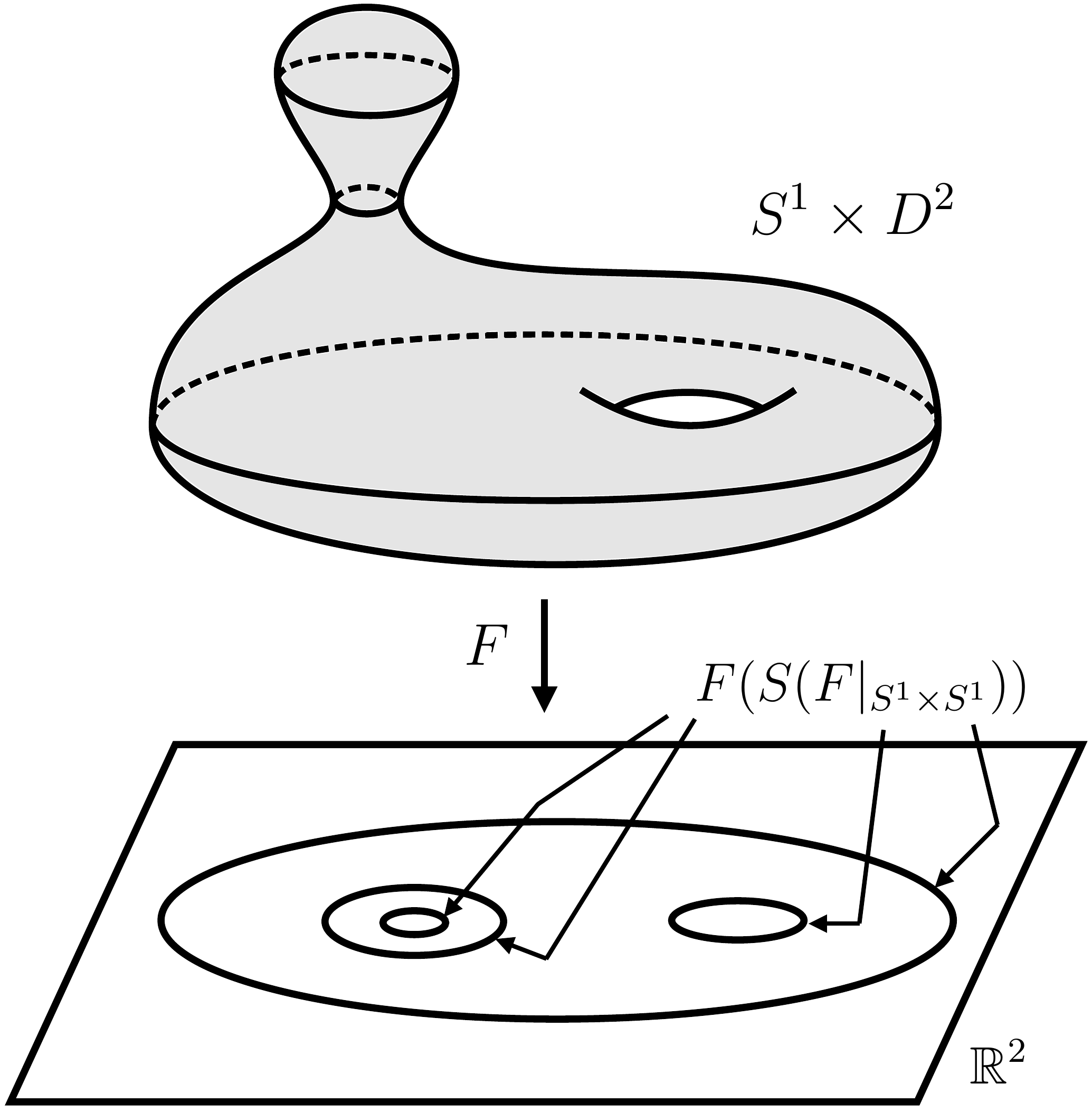}
      \caption{Submersion with definite folds $F$}
      \label{Fig3.1}
    \end{subfigure}%
    \hspace{8mm}%
    \raisebox{0mm}{%
      \begin{subfigure}[b]{0.25\textwidth}
        \centering
        \includegraphics[width=\textwidth]{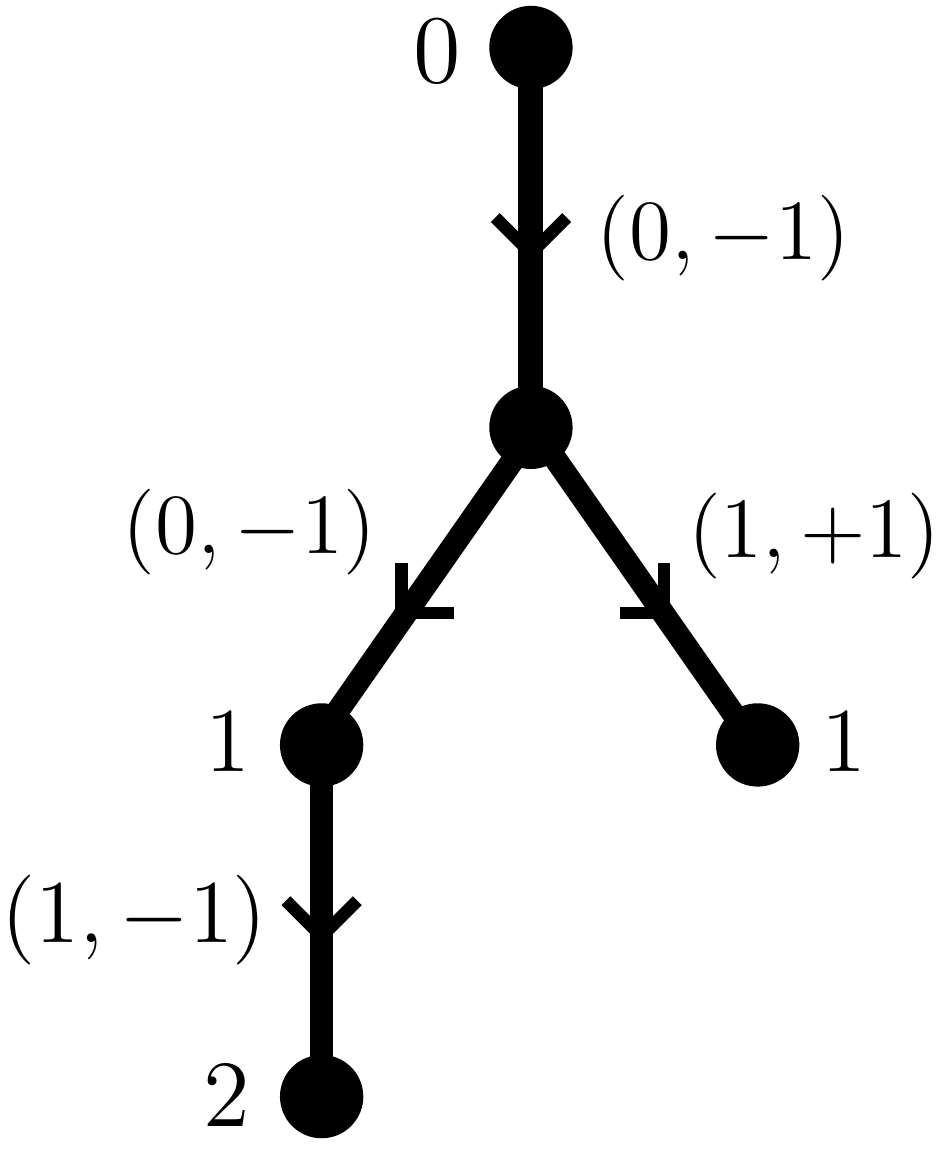}
        \caption{The graph $G_F$.}
        \label{Fig3.2}
      \end{subfigure}%
    }
  }
    \caption{The map $F\colon S^1\times D^2\to\mathbb{R}^2$ and its graph $G_F$}
  \label{Fig3}
\end{figure}
\end{example}

In particular, we have the following formula modulo $2$.

\begin{cor}\label{cor3.15}
Let $N$ be a compact, connected, $n$-dimensional manifold with boundary, and let $F\colon N\to\mathbb{R}^k$ be a submersion with definite folds whose restriction to the boundary is an image simple fold map, where $n>k$. 
Then, we have 
$$
\chi(N)
\equiv
\sum_{e\in E}\big
(\chi(R_e)
+
\chi(S_e)\big)d(r_e)
+
\sum_{e\in E(0,+)\sqcup E(n-k,+)}\chi(S_e)
\mod 2, 
$$
where $d(r_e)$ is the depth of the vertex corresponding to $R_e\ni r_e$. 
\end{cor}

By further using a property of definite fold maps, we have the following result.

\begin{cor}
Let $N$ be a compact, $(2l+1)$-dimensional manifold and let $F\colon N\to\mathbb{R}^3$ be a submersion with definite folds whose restriction to the boundary is an image simple fold map, where $l\geq 2$. 
If $\partial N$ is simply connected, then we have
$$
\chi(N)\equiv\sum_{e\in E}\chi(R_e)d(r_e) \mod 2, 
$$
where the definitions of $R_e$ and $r_e$ are the same as in \Cref{thm3.14}. 
\end{cor}

\begin{proof}
By the assumptions, it follows from \cite{Sak} that each $S_e$ is a $2$-dimensional sphere. 
Thus, we have the desired formula by \Cref{cor3.15}. 
\end{proof}

The following results on non-singular extensions follow immediately from \Cref{thm3.14}.

\begin{cor}\label{cor3.17}
Let $M$ be an $m$-dimensional closed manifold and let $f\colon M\to\mathbb{R}^k$ be an image simple fold map whose singular points are definite fold points, where $m\geq k$ and $m$ is even. 
If $f$ admits a non-singular extension $F\colon N\to\mathbb{R}^k$, then we have
\begin{align*}
\frac{1}{2}\chi(M)=\sum_{e\in E(0,+)} \chi(S_e) + \sum_{e\in E(m-k,+)} (-1)^{m-k+1}\chi(S_e)+ \sum_{e\in E} (\chi(R_e)-\chi(S_e))\chi(F^{-1}(r_e)).
\end{align*}
where the notation is the same as in \Cref{thm3.14}. 
\end{cor}

\begin{proof}
Since $m$ is even, we have $\chi(M)=2\chi(N)$. 
Using this relation and \Cref{thm3.14}, we obtain the desired equation. 
\end{proof}

\section*{Aknowledgement}
The author would like to thank Osamu Saeki, Noriyuki Hamada, and Naoki Kitazawa for their useful discussions. 
Furthermore, the author appreciates the reviewer for the fruitful comments. 
This work has been partially supported by JSPS KAKENHI Grant Number 26KJ1778, and by WISE program (MEXT) at Kyushu University.

\end{document}